\documentclass[a4paper]{amsart}
\usepackage{amssymb}
\usepackage[mathcal]{euscript}

\newcommand*{\mb}[1]{\mathbf{#1}}

\newcommand*{\supp}{\mathrm{supp}}

\newcommand*{\cJ}{\mathcal{J}}
\newcommand*{\cC}{\mathcal{C}}
\newcommand*{\cH}{\mathcal{H}}
\newcommand*{\cA}{\mathcal{A}}
\newcommand*{\cF}{\mathcal{F}}

\newcommand*{\frg}{{\mathfrak g}}
\newcommand*{\frh}{{\mathfrak h}}
\newcommand*{\fre}{{\mathfrak e}}
\newcommand*{\frf}{{\mathfrak f}}
\newcommand*{\frc}{{\mathfrak c}}
\newcommand*{\frd}{{\mathfrak d}}
\newcommand*{\frb}{{\mathfrak b}}
\newcommand*{\frs}{{\mathfrak s}}
\newcommand*{\frsl}{{\mathfrak{sl}}}
\newcommand*{\frso}{{\mathfrak{so}}}
\newcommand*{\frsp}{{\mathfrak{sp}}}

\newcommand*{\rank}{\mathrm{rank}} 

\newcommand*{\ad}{\mathrm{ad}}

\newcommand*{\NN}{\mathbb{N}}
\newcommand*{\ZZ}{\mathbb{Z}}
\newcommand*{\RR}{\mathbb{R}}
\newcommand*{\OO}{\mathbb{O}}
\newcommand*{\FF}{\mathbb{F}}

\newcommand*{\bydef}{:=}

\DeclareMathOperator{\Hom}{\mathrm{Hom}}
\DeclareMathOperator{\AAut}{\mathsf{Aut}}

\newtheorem{theorem}{Theorem}[section]
\newtheorem{proposition}[theorem]{Proposition}
\newtheorem{lemma}[theorem]{Lemma}
\newtheorem{corollary}[theorem]{Corollary}

\theoremstyle{definition}
\newtheorem{definition}[theorem]{Definition}
\newtheorem{example}[theorem]{Example}

\theoremstyle{remark}
\newtheorem{remark}[theorem]{Remark}

\def\hregla{\hrule height.1pt}
\def\hreglabis{\hrule height .3pt depth -.2pt}
\def\hregleta{\hrule height .5pt}
\def\hreglon{\hrule height1pt}
\def\vreglon{\vrule height 12pt width1pt depth 4pt}
\def\vregleta{\vrule width .5pt}

\def\hreglafill{\leaders\hreglabis\hfill}
\def\hregletafill{\leaders\hregleta\hfill}
\def\vregla{\vrule width.1pt}

\begin{document}

\title{Codes, $S$-structures, and exceptional Lie algebras}

\author[Isabel Cunha]{Isabel Cunha${}^\diamond$}
\address{Departamento de Matem\'atica e Centro de Matem\'atica e Aplica\c{c}\~{o}es da Universidade da Beira Interior, Universidade da Beira Interior, 6201-001 Covilh\~{a}, Portugal}
\email{icunha@ubi.pt}
\thanks{${}^\diamond$ 
Supported by Funda\c{c}\~{a}o para a Ci\^{e}ncia e Tecnologia, project PEst-OE/MAT/UI0212/2013}

\author[Alberto Elduque]{Alberto Elduque${}^\star{}^\dagger$}
\address{Departamento de Matem\'{a}ticas
 e Instituto Universitario de Matem\'aticas y Aplicaciones,
 Universidad de Zaragoza, 50009 Zaragoza, Spain}
\email{elduque@unizar.es}
\thanks{${}^\star$ Supported by grants MTM2017-83506-C2-1-P (AEI/FEDER, UE) and E22\_17R (Diputaci\'on General de Arag\'on)}
\thanks{${}^\dagger$ Corresponding author}




\begin{abstract}
The exceptional simple Lie algebras of types $E_7$ and $E_8$ are endowed with optimal $\mathsf{SL}_2^n$-structures, and are thus described in terms of the corresponding coordinate algebras. These are nonassociative algebras which much resemble the so called code algebras.
\end{abstract}

\maketitle

\section{Introduction}\label{se:intro}

There is a well-known connection between binary linear codes and root lattices. The reader may refer to \cite{CS} or \cite{Ebeling}. In particular, the $E_8$ root lattice is obtained from the extended Hamming $[8,4,4]$ binary linear code, and the $E_7$ root lattice from the simplex $[7,3,4]$ binary linear code, dual to Hamming's $[7,4,3]$ code.

Recently, a new class of commutative nonassociative (i.e., not necessarily associative) algebras have been defined in \cite{CMR}. They are called \emph{code algebras}. These algebras contain a family of orthogonal idempotents and a nice `Peirce decomposition' relative to this family. Code algebras are inspired by some axiomatic approaches to Vertex Operator Algebras.

On the other hand, Vinberg \cite{Vinberg} introduced lately the notion of $\mathsf{S}$-structure in a Lie algebra, as an extension of the notion of grading by an abelian group. Given an $\mathsf{S}$-structure in a Lie algebra, the isotypic decomposition relative to the action of the reductive group $\mathsf{S}$ provides a description of the Lie algebra in terms of a nonassociative system (algebra, pair, triple system, ...) that \emph{coordinatizes} the Lie algebra. Something similar happens for root graded Lie algebras, a subject initiated by Berman and Moody \cite{BermanMoody}. In the $BC_r$-case \cite{AllisonBenkartGao}, the isotypic decomposition becomes more involved, with the possibilities for the grading subalgebra to be of type $B$, $C$, or $D$, and there appear several different coordinate algebras.

Finally, the Lie algebras in Freudenthal's Magic Square, that includes the exceptional simple Lie algebras of types $F_4$, $E_6$, $E_7$, and $E_8$, were described in \cite{Eld07} in terms of very simple components, copies of the $3$-dimensional simple Lie algebra and of its $2$-dimensional simple representation. These descriptions were obtained by constructing the Lie algebras in Freudenthal's Magic Square by means of a couple of symmetric composition algebras and their triality Lie algebras, and by describing these simpler objects in the above terms. 

It turns out that a closer look at the results in \cite{Eld07} shows that these can be recast in terms of \emph{optimal short $\mathsf{SL}_2^n$-structures} in the corresponding Lie algebras. And this is the goal of this paper. 

For the exceptional simple Lie algebras of types $E_7$ and $E_8$, the coordinate algebras that appear are quite close to the code algebras in \cite{CMR}, although commutativity is not assumed in the former ones. Not surprisingly, the codes involved are the simplex and the extended Hamming binary linear codes mentioned above. Moreover, the structure constants can be described in terms of the real division algebra of the octonions. In this way, not only the root lattices of types $E_7$ and $E_8$ are described in terms of these codes, but their Lie brackets are determined by these codes and a sort of `code algebras' attached to them.

The paper is structured as follows. Section \ref{se:codes_lattices} reviews the connection between binary linear codes and root lattices, and presents the codes that will appear throughout the paper. Vinberg's $\mathsf{S}$-structures are recalled in Section \ref{se:Sstructures}, and the (optimal) short $\mathsf{SL}_2^n$-structures are introduced. A simple Lie algebra over an algebraically closed field of characteristic $0$ admits such an optimal structure if and only if it is of one of the following types (Corollary \ref{co:optimal}): $A_1$, $B_{2n}$ $(n\geq 2)$, $C_n$ ($n\geq 2$), $D_{2n}$ ($n\geq 2$), $E_7$, $E_8$, or $F_4$. The corresponding coordinate algebras are studied too. Sections \ref{se:e7} and \ref{se:e8} are devoted to recast the results in \cite{Eld07} about the exceptional simple Lie algebras of types $E_7$ and $E_8$ in terms of optimal $\mathsf{SL}_2^n$-structures and their coordinate algebras. These are described in terms of the real division algebra of the octonions, following the ideas in \cite{Eld07}. The coordinate algebras in these cases are quite close to code algebras. Section \ref{se:f4} gives similar results for the exceptional simple Lie algebra of type $F_4$, while Section \ref{se:classical} outlines the description of the optimal $\mathsf{SL}_2^n$-structure for the remaining classical cases in Corollary \ref{co:optimal}.

\bigskip

\section{Codes and root systems}\label{se:codes_lattices}

The goal of this section is to review the connections between binary linear codes and root lattices, and to provide suitable examples.

Denote by $\mb{u}\bullet\mb{v}$ the standard dot product in $\RR^n$: $(u_1,\ldots,u_n)\bullet(v_1,\ldots,v_n)=u_1v_1+\cdots + u_nv_n$.

Let $\Gamma\subset \RR^n$ be an even lattice, i.e., a lattice such that $\mb{x}^{\bullet 2}\in 2\ZZ$ for all $\mb{x}\in\Gamma$. The \emph{roots} of $\Gamma$ are the elements $\mb{x}\in\Gamma$ such that $\mb{x}^{\bullet 2}=2$. The even lattice $\Gamma$ is said to be a \emph{root lattice} if its set of roots spans $\Gamma$.

Every root lattice $\Gamma$ is the orthogonal direct sum of the irreducible root lattices corresponding to the simply laced Dynkin diagrams $A_n$ ($n\geq 1$), $D_n$ ($n\geq 4$), $E_6$, $E_7$, $E_8$. (See \cite[Theorem 1.2]{Ebeling}.)  Its Weyl group $W(\Gamma)$ is the group generated by the reflections at the hyperplanes orthogonal to all roots.

Consider a binary linear code $\mb{C}\subseteq \FF_2^n$, that is, a vector subspace of $\FF_2^n$, where $\FF_2$ denotes the field of two elements,  and consider the reduction modulo $2$ map
\[
\rho:\ZZ^n\longrightarrow \left(\ZZ/2\right)^n=\FF_2^n\,.
\]
This is a group homomorphism and $\Gamma_{\mb{C}}\bydef\frac{1}{\sqrt{2}}\rho^{-1}(\mb{C})$ is a lattice in $\RR^n$.

\begin{proposition}[{see \cite[Proposition 1.5]{Ebeling}}]\label{pr:code_lattices}
Let $\Gamma\subset\RR^n$ be an irreducible root lattice. Then the following statements are equivalent:
\begin{enumerate}
\item[(i)] $\Gamma=\Gamma_{\mb{C}}$ for a binary linear code $\mb{C}\subseteq \FF_2^n$.
\item[(ii)] $\Gamma$ contains $n$ pairwise orthogonal roots.
\item[(iii)] $nA_1=A_1\oplus\cdots\oplus A_1$ is a sublattice of $\Gamma$.
\item[(iv)] $-1\in W(\Gamma)$.
\item[(v)] $2\Gamma^*\subseteq \Gamma$, where $\Gamma^*$ is the \emph{dual lattice} $\Gamma^*\bydef\{\mb{x}\in\RR^n\mid \mb{x}\bullet\mb{y}\in\ZZ\ \forall \mb{y}\in\Gamma\}$.
\item[(vi)] $\Gamma$ is of type $A_1$, $D_{2n}$ ($n\geq 2$), $E_7$ or $E_8$.
\end{enumerate}
\end{proposition}

The binary linear codes in the following examples provide the key codes and root lattices for the remaining of the paper.

\begin{example}\label{ex:Hamming}
The Hamming $[7,4,3]$ binary linear code is defined on $\FF_2^7$ by the parity check relations:
\[
\begin{split}
c_1+c_3+c_5+c_7&=0\\
c_2+c_3+c_6+c_7&=0\\
c_4+c_5+c_6+c_7&=0
\end{split}
\]
We may think of $c_3,c_5,c_6,c_7$ as the data bits, and $c_1,c_2,c_4$ as the check bits. Note that, written in base $2$, the indices $1,3,5,7$ in the first equation are the natural numbers $\leq 7$ with the last binary digit equal to $1$, the indices $2,3,6,7$ in the second equation are those with the second last digit $1$, and the indices $4,5,6,7$ in the third equation are those with the first digit $1$.

Hence, a \emph{check matrix} for the Hamming code is
\begin{equation}\label{eq:check_Hamming}
\begin{pmatrix}
1&0&1&0&1&0&1\\ 0&1&1&0&0&1&1\\ 0&0&0&1&1&1&1
\end{pmatrix}
\end{equation}
A \emph{generator matrix} is given by (using $c_3,c_5,c_6,c_7$ as the data bits):
\begin{equation}\label{eq:generator_Hamming}
\begin{pmatrix}
1&1&1&0&0&0&0\\
1&0&0&1&1&0&0\\
0&1&0&1&0&1&0\\
1&1&0&1&0&0&1
\end{pmatrix}
\end{equation}
\end{example}

\begin{example}\label{ex:simplex}
The \emph{simplex} $[7,3,4]$ binary linear code $\mb{C}$ is the dual of the previous Hamming code, a generator matrix is then being given by \eqref{eq:check_Hamming}. For later use, let us permute columns $i\leftrightarrow 8-i$ in \eqref{eq:check_Hamming}, so here the simplex code will have generator matrix:
\begin{equation}\label{eq:simplex}
\begin{pmatrix}
1&0&1&0&1&0&1\\
1&1&0&0&1&1&0\\
1&1&1&1&0&0&0
\end{pmatrix}
\end{equation}
The corresponding root lattice is $E_7$ \cite[p.~26]{Ebeling}.
\end{example}

\begin{example}\label{ex:extended_Hamming}
We add one extra dimension to the Hamming code and the extra global parity check 
\[
c_0+c_1+c_2+c_3+c_4+c_5+c_6+c_7=0
\]
to get the \emph{extended Hamming $[8,4,4]$ binary linear code} with generator matrix:
\begin{equation}\label{eq:extended_Hamming}
\begin{pmatrix}
1&1&1&1&0&0&0&0\\
1&1&0&0&1&1&0&0\\
1&0&1&0&1&0&1&0\\
0&1&1&0&1&0&0&1
\end{pmatrix}
\end{equation}
The corresponding root lattice is $E_8$ \cite[\S 1.3]{Ebeling}
\end{example}

A particular class of algebras associated to binary linear codes has been defined in \cite[Definition 1]{CMR}:

\begin{definition}\label{df:code_algebra}
Let $\mb{C}\subseteq \FF_2^n$ be a binary linear code. A \emph{code algebra} based on $\mb{C}$ is a commutative algebra over a field $\FF$, endowed with a basis 
\[
\{t_i\mid i=1,\ldots,n\}\cup\left\{e^{\mb{c}}\mid \mb{c}\in\mb{C}\setminus\{\mb{0},\mb{1}\}\right\}
\]
(here $\mb{0}=(0,0,\ldots,0)$ and $\mb{1}=(1,1,\ldots,1)$), that satisfies the following relations:
\[
\begin{split}
&t_it_j=\begin{cases} t_i&\text{if $i=j$,}\\ 0&\text{otherwise,}\end{cases}\\
&t_ie^{\mb{c}}\in \FF e^{\mb{c}},\\
&e^{\mb{c}}e^{\mb{d}}\in\FF e^{\mb{c}+\mb{d}},\ \text{for $\mb{c}\neq \mb{d},\mb{1}-\mb{d}$,}\\
&\left(e^{\mb{c}}\right)^2\in\sum_{i\in\supp(\mb{c})}\FF t_i,\\
&e^{\mb{c}}e^{\mb{1}-\mb{c}}=0,
\end{split}
\]
for $1\leq i,j\leq n$ and $\mb{c},\mb{d}\in\mb{C}\setminus\{\mb{0},\mb{1}\}$, where for $\mb{c}=(c_1,\ldots,c_n)$, its \emph{support} $\supp(\mb{c})$ denotes the set of indices with $c_i=1$. (Thus, for instance, $\supp\bigl((1,0,1,0)\bigr)=\{1,3\}$.)
\end{definition}

The \emph{coordinate algebras} obtained in Sections \ref{se:e7} and \ref{se:e8}, relative to the simple Lie algebras of types $E_7$ and $E_8$, satisfy all the restrictions to be a code algebra, except for their lack of commutativity. The one obtained in Section \ref{se:f4}, relative to $F_4$, is a bit different, as an extra basic element $e^{\mb{1}}$ is included.

\bigskip

\section{$\mathsf{S}$-structures}\label{se:Sstructures}

Assume, throughout the paper, that our ground field $\FF$ is algebraically closed of characteristic $0$. Unadorned tensor products will be understood to be defined over $\FF$. All algebras will be assumed to be finite dimensional.

Given a finitely generated abelian group, a $G$-grading on a nonassociative algebra $\cA$ corresponds to a morphism of affine group schemes $G^D\longrightarrow \AAut(\cA)$, where $G^D$ is the Cartier dual of $G$, represented by the group algebra $\FF G$, which is a quasitorus (see \cite[Chapter 1]{EKmon}). In this situation, the homogeneous components are the isotypic components of the action of $G^D$. Vinberg \cite{Vinberg} introduced \emph{non-abelian gradings} as isotypic decompositions with respect to reductive groups of automorphisms. More specifically, given a reductive algebraic group $\mathsf{S}$, extending Vinberg's definition to arbitrary nonassociative algebras we get the following:

\begin{definition}[{see \cite[Definition 0.1]{Vinberg}}]\label{df:S_structure}
An $\mathsf{S}$-structure in a nonassociative algebra $\cA$ is a homomorphism $\Phi:\mathsf{S}\rightarrow \AAut(\cA)$.
\end{definition}

In this situation, the differential $\mathrm{d}\Phi$ gives a representation of the Lie algebra of $\mathsf{S}$ as derivations on $\cA$, $\mathrm{d}\Phi:\mathrm{Lie}(\mathsf{S})\rightarrow \mathfrak{der}(\cA)$.

In \cite{Vinberg} two different specific types of $\mathsf{S}$-structures are considered: very short $\mathsf{SL}_2$-structures and short $\mathsf{SL}_3$-structures in a Lie algebra $\frg$.

A nontrivial $\mathsf{SL}_2$-structure $\Phi$ in a Lie algebra $\frg$ is called \emph{very short} if the representation $\Phi$ decomposes into $1$- and $3$-dimensioonal irreducible representations. In a semisimple Lie algebra $\frg$, a very short $\mathsf{SL}_2$-structure gives rise to an isotypic decomposition of the form
\[
\frg=\bigl(\frsl_2\otimes \cJ\bigr)\oplus\mathfrak{der}(\cJ)
\]
for a semisimple Jordan algebra $\cJ$. (See \cite{Vinberg} and the references therein.)

Similarly, a nontrivial $\mathsf{SL}_3$-structure in a simple Lie algebra $\frg$ is called \emph{short} if the representation $\Phi$ decomposes into the adjoint representation of $\mathsf{SL}_3$ and $1$- and $3$-dimensional irreducible representations. (A more general situation is considered in \cite{BE}.) 
In this case, the Lie algebra $\frg$ can also be described in terms of a cubic Jordan algebra $\cJ$ \cite[Equation (31)]{Vinberg}:
\[
\frg=\frsl_3\oplus(V\otimes \cJ)\oplus(V^*\otimes \cJ)\oplus\mathfrak{str}_0(\cJ).
\]

\bigskip

\subsection{$\mathsf{SL}_2^n$-structures} \null\quad

Here we will consider the following $\mathsf{S}$-structures:

\begin{definition}\label{df:SL2n_structures}
Let $\frg$ be a simple Lie algebra, and let $n\in\NN$.
\begin{itemize}
\item An $\mathsf{SL}_2^n$-structure $\Phi:\mathsf{SL}_2^n\longrightarrow \AAut(\frg)$ is called \emph{short} if the representation $\Phi$ decomposes into the adjoint representation of $\mathsf{SL}_2^n$, irreducible representations formed by tensor products of the $2$-dimensional natural representations of some of the copies  of $\mathsf{SL}_2$ (without repetitions), and $1$-dimensional representations.

\item A short $\mathsf{SL}_2^n$-structure is said to be \emph{optimal} if $n=\rank(\frg)$.
\end{itemize}
\end{definition}

Given an $\mathsf{SL}_2^n$-structure in a Lie algebra $\frg$, let $V_i$ be the $2$-dimensional irreducible representation for the $i^{\text{th}}$ factor in $\mathsf{SL}_2^n$. Given any $\mb{c}\in\FF_2^n$, denote by $V^{\mb{c}}$ the $\mathsf{SL}_2^n$-module obtained as the tensor product of the $V_i$'s with $i\in\supp(\mb{c})$. Thus, for example, with $n=8$ and $\mb{c}=(1,0,0,1,0,1,1,0)$,
\[
V^{\mb{c}}=V_1\otimes V_4\otimes V_6\otimes V_7.
\]
In particular, $V^{\mb{0}}=\FF$ is the $1$-dimensional trivial representation.

\smallskip

Note that if $\Phi:\mathsf{SL}_2^n\longrightarrow\AAut(\frg)$ is a short $\mathsf{SL}_2^n$-structure in the simple Lie algebra $\frg$, the differential $\mathrm{d}\Phi$ gives a Lie algebra homomorphism $\mathrm{d}\Phi:\frsl_2^n\longrightarrow\frg\simeq\mathfrak{der}(\frg)$, and $\mathrm{d}\Phi$ is one-to-one, as the adjoint representation is a component of the representation $\Phi$.

Therefore, if $\Phi:\mathsf{SL}_2^n\longrightarrow \AAut(\frg)$ is a short $\mathsf{SL}_2^n$-structure in $\frg$, the isotypic decomposition of $\frg$ is of the form:
\begin{equation}\label{eq:isotypic}
\frg=\frsl_2^n\oplus\Bigl(\bigoplus_{\mb{c}\in\FF_2^n\setminus\{\mb{0}\}}(V^{\mb{c}}\otimes\cA^{\mb{c}})\Bigr)\oplus\frc
\end{equation}
where the subalgebra $\frsl_2^n$ (the image of $\mathrm{d}\Phi$) is the adjoint representation of $\mathsf{SL}_2^n$, the $\cA^{\mb{c}}$'s are vector spaces whose dimension indicates the multiplicity of $V^{\mb{c}}$, and $\frc$ is the sum of the $1$-dimensional representations, so that $\frc$ is the centralizer in $\frg$ of the subalgebra $\frsl_2^n$ and, as such, it is a subalgebra of $\frg$.

For each $i=1,\ldots,n$, let $\{e_i,f_i,h_i\}$ be a standard basis of the $i^{\text{th}}$ copy of $\frsl_2$: $[h_i,e_i]=2e_i$, $[h_i,f_i]=-2f_i$, $[e_i,f_i]=h_i$.

Hence, a short $\mathsf{SL}_2^n$-structure in the simple Lie algebra $\frg$ is given by a subalgebra of $\frg$ isomorphic to $\frsl_2^n$, such that the eigenvalues of the adjoint map $\ad h_i$ are $\pm 2$ with multiplicity $1$, and $\pm 1$ and $0$, because the eigenvalues of $\ad h_i$ on $V_i$ are $\pm 1$. 

The subspace $\FF h_1\oplus\cdots\oplus \FF h_n$ is a toral subalgebra of $\frg$, and hence contained in a Cartan subalgebra, which has the following form:
\begin{equation}\label{eq:h_cartan}
\frh=\FF h_1\oplus\cdots\oplus\FF h_n\oplus (\frh\cap \frc).
\end{equation}
Then the linear map $\alpha_i:\frh\rightarrow \FF$ given by
\begin{equation}\label{eq:root}
\alpha_i(h_i)=2,\quad \alpha_i(h_j)=0\ \text{if $i\neq j$},\quad \alpha_i(\frh\cap\frc)=0,
\end{equation}
is a root of $\frh$ with root space $\frg_{\alpha_i}=\FF e_i$.

\begin{theorem}\label{th:SL2n_structures}
Let $\Phi:\mathsf{SL}_2^n\longrightarrow\AAut(\frg)$ be a short $\mathsf{SL}_2^n$-structure in the simple Lie algebra $\frg$. Let $\{e_i,f_i,h_i\}$ be a standard basis of the image under $\mathrm{d}\Phi$ of the $i^{\text{th}}$ copy of $\frsl_2$. Let $\frh$ be the Cartan subalgebra in \eqref{eq:h_cartan}, and let $\alpha_1,\ldots,\alpha_n$ be the roots, relative to $\frh$, defined in \eqref{eq:root}. Then $\{\alpha_1,\ldots,\alpha_n\}$ is a set of pairwise orthogonal long roots.

Conversely, if $\frh$ is a Cartan subalgebra of the simple Lie algebra $\frg$ with associated root system $R$, and if $\{\alpha_1,\ldots,\alpha_n\}$ is a set of pairwise orthogonal roots in $R$, then 
\[
\frs_i=\frg_{\alpha_i}\oplus\frg_{-\alpha_i}\oplus [\frg_{\alpha_i},\frg_{-\alpha_i}]
\]
is a Lie subalgebra isomorphic to $\frsl_2$, $[\frs_i,\frs_j]=0$ for $i\neq j$ and the embedding
\[
\frsl_2^n\simeq \frs_1\oplus\cdots\oplus\frs_n\hookrightarrow \frg
\]
integrates to a short $\mathsf{SL}_2^n$-structure.
\end{theorem}
\begin{proof}
Note that given a root $\alpha$ relative to a Cartan subalgebra $\frh$ of the simple Lie algebra $\frg$, $\alpha$ is long if and only if $\langle \beta\mid\alpha\rangle =2\dfrac{(\beta\mid\alpha)}{(\alpha\mid\alpha)}\in \{0,1,-1\}$ for any root $\beta\neq\pm\alpha$. (Here, as usual, $(.\mid.)$ denotes the symmetric bilinear form on $\frh^*$ induced by the Killing form, or a scalar multiple of it.) With $h_\alpha\in [\frg_\alpha,\frg_{-\alpha}]$ such that $\alpha(h_\alpha)=2$, this means that $\beta(h_\alpha)\in \{0,1,-1\}$ for any root $\beta\neq \pm\alpha$. Thus the eigenvalues of $\ad h_\alpha$ are $\pm 2$ with multiplicity $1$, and $\pm 1$ and $0$.

Now, if $\Phi:\mathsf{SL}_2^n\longrightarrow \AAut(\frg)$ is a short $\mathsf{SL}_2^n$-structure in $\frg$, this shows that the $\alpha_i$'s defined in \eqref{eq:root} are orthogonal long roots.

Conversely, if $\{\alpha_1,\ldots,\alpha_n\}$ is a set of pairwise orthogonal long roots relative to a Cartan subalgebra $\frh$, consider the subalgebras $\frs_i=\frg_{\alpha_i}\oplus\frg_{-\alpha_i}\oplus [\frg_{\alpha_i},\frg_{-\alpha_i}]$, $i=1,\ldots,n$. Let $h_i\in [\frg_{\alpha_i},\frg_{-\alpha_i}]$ be the element with $\alpha_i(h_i)=2$. Then $\overline{\frh}=\FF h_1\oplus\cdots\oplus\FF h_n$ is a Cartan subalgebra of $\frs_1\oplus\cdots\oplus\frs_n\simeq \frsl_2^n$, and with $\lambda_i:\overline{\frh}\rightarrow \FF$ given by $\lambda_i(h_i)=1$, $\lambda_i(h_j)=0$ if $i\neq j$, the weights of the adjoint representation of $\frs_1\oplus\cdots\oplus\frs_n$ in $\frg$ are the $\pm 2\lambda_i$'s with multiplicity $1$, $0$, and weights of the form $\pm \lambda_{i_1}\pm\cdots\pm\lambda_{i_r}$, $1\leq i_1<\cdots <i_r\leq n$. Hence this adjoint representation decomposes into the adjoint module $\frs_1\oplus\cdots \oplus\frs_n$, irreducible representations formed by tensor products of the $2$-dimensional irreducible representations for each $\frs_i$, and $1$-dimensional representations. Therefore, this adjoint representation integrates to a short $\mathsf{SL}_2^n$-structure.
\end{proof}

\begin{corollary}\label{co:optimal}
Let $\frg$ be a simple Lie algebra. Then $\frg$ admits an optimal short $\mathsf{SL}_2^n$-structure if and only if $\frg$ is of type $A_1$, $B_{2n}$ $(n\geq 2)$, $C_n$ ($n\geq 2$), $D_{2n}$ ($n\geq 2$), $E_7$, $E_8$, or $F_4$. Any two optimal short $\mathsf{SL}_2^n$-structures of $\frg$ are conjugate by an automorphism.
\end{corollary}
\begin{proof}
Any two maximal subsets of orthogonal long roots of an irreducible root system are conjugate under the Weyl group \cite[Chapter VI, \S 1, Exercise 14]{Bou4_6}. As any two Cartan subalgebras of $\frg$ are conjugate and any element of the Weyl group relative to a Cartan subalgebra lifts to an automorphism of $\frg$ \cite[Chapter VIII, \S 5.2]{Bou7_9}, the conjugacy result follows.

Now, $\frg$ admits an optimal short $\mathsf{SL}_2^n$-structure if and only if the root system consisting of the long roots of $\frg$ relative to a Cartan subalgebra is of type $nA_1$, $D_{2n}$ ($n\geq 2$), $E_7$ or $E_8$ by Proposition \ref{pr:code_lattices}. But the system of long roots of $B_n$ ($n\geq 3$) is $D_n$, that of $C_n$ is $nA_1$, that of $G_2$ is $A_2$, and that of $F_4$ is $D_4$. Thus, due to Proposition \ref{pr:code_lattices}, $\frg$ admits an optimal $\mathsf{SL}_2^n$-structure if and only if $\frg$ is of type $A_1$, $B_{2n}$ ($n\geq 2$), $C_n$ ($n\geq 2$), $D_{2n}$ ($n\geq 2$), $E_7$, $E_8$, or $F_4$.
\end{proof}

\medskip

\subsection{Coordinate algebra}\label{ss:coordinate} \null\quad

Given an optimal short $\mathsf{SL}_2^n$-structure in a simple Lie algebra $\frg$, we get the isotypic decomposition \eqref{eq:isotypic} with $\frc=0$ (as $n=\rank(\frg)$, and hence $\frc$ centralizes the Cartan subalgebra $\FF h_1\oplus\cdots \oplus\FF h_n$, which is its own centralizer), and with $\dim \cA^\mb{c}\leq 1$ for any $\mb{c}\in\FF_2^n\setminus\{\mb{0}\}$, because the multiplicity of any root is $1$. The isotypic decomposition \eqref{eq:isotypic} may be then rewritten as:
\begin{equation}\label{eq:isotypic2}
\frg=\Bigl(\bigoplus_{i=1}^n(\frsl(V_i)\otimes\FF t_i)\Bigr)\oplus\Bigl(\bigoplus_{\mb{c}\in \mb{S}}(V^{\mb{c}}\otimes \FF e^{\mb{c}})\Bigr)
\end{equation}
for some subset $\mb{S}\subseteq \FF_2^n\setminus\{\mb{0}\}$.

The next result is a straightforward consequence of Clebsch-Gordan formula \cite[Chapter VIII, \S 9.4]{Bou7_9}:

\begin{lemma}\label{le:sl_invariants}
Let $V$ be a $2$-dimensional vector space endowed with a nonzero skew-symmetric bilinear form $\langle.\mid.\rangle$. Then:
\begin{itemize}
\item $\Hom_{\mathsf{SL}(V)}(V\otimes V,\FF)$ is spanned by $\langle .\mid.\rangle$.

\item $\Hom_{\mathsf{SL}(V)}(\frsl(V)\otimes V,V)$ is spanned by the natural action of $\frsl(V)$ on $V$.

\item $\Hom_{\mathsf{SL}(V)}(V\otimes V,\frsl(V))$ is spanned by the map $u\otimes v\mapsto \Bigl(s_{u,v}:w\mapsto \frac{1}{2}\bigl(\langle w\mid u\rangle v+\langle w\mid v\rangle u\bigr)\Bigr)$. (This is $-\frac{1}{2}\gamma_{u,v}$, for $\gamma_{u,v}$ in \cite[(2.1)]{Eld07}.)

\item $\Hom_{\mathsf{SL}(V)}(\frsl(V)\otimes\frsl(V),V)=0=\Hom_{\mathsf{SL}(V)}(V\otimes V,V)$.
\end{itemize}
\end{lemma}

Equip the $V_i$'s above with a fixed nonzero skew-symmetric bilinear map $\langle .\mid.\rangle$ (the same notation is used for all $i$). Lemma \ref{le:sl_invariants} tells us that given $\mb{c}\neq\mb{d}\in\FF_2^n\setminus\{\mb{0}\}$, there is a unique, up to scalars, nonzero linear map 
\[
\varphi_{\mb{c},\mb{d}}:V^{\mb{c}}\times V^{\mb{d}}\longrightarrow V^{\mb{c}+\mb{d}}
\]
invariant under the action of $\mathsf{SL}_2^n=\mathsf{SL}(V_1)\times\cdots\times \mathsf{SL}(V_n)$, and this is given by contraction on the `common indices'. Thus, for instance,
\[
\varphi_{(1,1,1,0),(1,0,1,1)}\bigl(u_1\otimes u_2\otimes u_3,v_1\otimes v_3\otimes v_4\bigr)=
\langle u_1\mid v_1\rangle\langle u_3\mid v_3\rangle u_2\otimes v_4,
\]
for $u_i,v_i\in V_i$, $1\leq i\leq 4$.

Similarly, for $\mb{c}\in\FF_2^n\setminus\{\mb{0}\}$ with $1$ in the $i^{\text{th}}$ position, there is a unique, up to scalars, nonzero linear map
\[
\varphi_{\mb{c},\mb{c}}^i:V^{\mb{c}}\times V^{\mb{c}}\longrightarrow \frsl(V^i),
\]
given by contraction on the indices different from $i$ and using $s_{u,v}$ in Lemma \ref{le:sl_invariants}. Thus, for instance,
\[
\varphi_{(1,1,1,0),(1,1,1,0)}^2\bigl(u_1\otimes u_2\otimes u_3,v_1\otimes v_2\otimes v_3)=
\langle u_1\mid v_1\rangle\langle u_3\mid v_3\rangle s_{u_2,v_2}.
\]

Consider the vector space
\[
\cC=\FF t_1\oplus\cdots \oplus\FF t_n\oplus\Bigl(\bigoplus_{\mb{c}\in\mb{S}}\FF e^{\mb{c}}\Bigr).
\] 
The invariance of the Lie bracket of $\frg$ under the action of $\mathsf{SL}_2^n=\mathsf{SL}(V_1)\times\cdots\times \mathsf{SL}(V_n)$ induces a bilinear multiplication on $\cC$ with:
\[
\begin{split}
& t_i^2=t_i,\ i=1,\ldots,n;\quad t_it_j=0\ \text{for $i\neq j$ (that is,
$t_it_j=\delta_{ij}$ for all $i,j$),}\\[6pt]
& t_ie^{\mb{c}}=e^{\mb{c}}t_i=\begin{cases} e^{\mb{c}}&\text{if $i\in\supp(\mb{c})$,}\\
     0&\text{otherwise,}\end{cases}\\[6pt]
& e^{\mb{c}}e^{\mb{d}}\in \FF e^{\mb{c}+\mb{d}}\quad\text{for $\mb{c}\neq \mb{d}$ in $\mb{S}$,}\\[6pt]
& e^{\mb{c}}e^{\mb{c}}\in\sum_{i\in\supp(\mb{c})}\FF t_i,\quad\text{for $\mb{c}\in\mb{S}$,}
\end{split}
\]
such that, if we write $e^{\mb{c}}e^{\mb{c}}=\sum_{i\in\supp(\mb{c})}\mu_i t_i$, $\mu_i\in \FF$ for all $i$, the Lie bracket on $\frg$ is given by:
\[
\begin{split}
&[x\otimes t_i,y\otimes t_i]=[x,y]\otimes t_i,\quad\text{for $i=1,\ldots,n$, $x,y\in\frsl(V_i)$,}
\\[6pt]
&[x\otimes t_i,y\otimes t_j]=0,\quad\text{for $1\leq i\neq j\leq n$, $x\in\frsl(V_i)$, $y\in\frsl(V_j)$,}
\\[6pt]
&[x\otimes t_i,(u_{i_1}\otimes\cdots\otimes u_{i_r})\otimes e^{\mb{c}}]=\\
&\qquad\qquad
  \begin{cases} 0&\text{if $i\not\in\supp(\mb{c})$,}\\
  (u_{i_1}\otimes\cdots\otimes (xu_{i_j})\otimes\cdots\otimes u_{i_r})\otimes e^{\mb{c}}&\text{if $i=i_j\in\supp(\mb{c})$,}
  \end{cases}
  \\
  &\qquad\qquad\text{for $\mb{c}\in\mb{S}$ with $\supp(\mb{c})=(i_1,\ldots,i_r)$, $u_{i_s}\in V_{i_s}$, $s=1,\ldots,r$, $x\in\frsl(V_i)$,}
  \\[6pt]
&[X\otimes e^{\mb{c}},Y\otimes e^{\mb{d}}]=\varphi_{\mb{c},\mb{d}}(X,Y)\otimes e^{\mb{c}}e^{\mb{d}},\quad
 \text{for $\mb{c}\neq\mb{d}$ in $\mb{S}$, $X\in V^{\mb{c}}$, $Y\in V^{\mb{d}}$,}
\\[6pt]
&[X\otimes e^{\mb{c}},Y\otimes e^{\mb{c}}]=\sum_{i\in\supp(\mb{c})}\mu_i\varphi_{\mb{c},\mb{c}}^i(X,Y)\otimes t_i,\quad\text{for $\mb{c}\in\mb{S}$, $X,Y\in V^{\mb{c}}$.}
\end{split}
\]

In other words, the Lie algebra $\frg$ is completely determined by its \emph{coordinate algebra} $\cC$.

\begin{remark}\label{re:root_system}
Given a standard basis $\{e_i,f_i,h_i\}$ of $\frsl(V_i)$, the subspace 
\[
\frh=\bigoplus_{i=1}^n\FF(h_i\otimes t_i)
\] 
is a Cartan subalgebra of $\frg$ in \eqref{eq:isotypic2}. Let, as before, $\alpha_i\in\frh^*$ be defined by $\alpha_i(h_i\otimes t_j)=2\delta_{ij}$. Then the root system of $\frg$ relative to $\frh$ is
\[
\Phi=\{\pm \alpha_i: i=1,\ldots,n\}\cup\Bigl(\bigcup_{\mb{c}\in\mb{S}}\Bigl\{\frac{1}{2}\sum_{i\in\supp(\mb{c})} (\pm\alpha_i)\Bigr\}\Bigr).
\]
\end{remark}

\bigskip

\section{Short $\mathsf{SL}_2^7$-structure on $E_7$ and the simplex $[7,4,3]$ binary linear code}\label{se:e7}

The simplex code (Example \ref{ex:simplex}) is the binary linear code $\mb{C}\subseteq\FF_2^7$ with 
\[
\mb{C}=\{\mb{0},\mb{c}_1,\mb{c}_2,\mb{c}_3,\mb{c}_4,\mb{c}_5,\mb{c}_6,\mb{c}_7\},
\]
where
\[
\begin{split}
\mb{c}_1&=(1,1,0,0,1,1,0),\\
\mb{c}_2&=(0,1,1,0,0,1,1),\\
\mb{c}_3&=(1,0,1,0,1,0,1)=\mb{c}_1+\mb{c}_2,\\
\mb{c}_4&=(1,1,1,1,0,0,0),\\
\mb{c}_5&=(0,0,1,1,1,1,0)=\mb{c}_1+\mb{c}_4,\\
\mb{c}_6&=(1,0,0,1,0,1,1)=\mb{c}_2+\mb{c}_4,\\
\mb{c}_6&=(0,1,0,1,1,0,1)=\mb{c}_3+\mb{c}_4.
\end{split}
\]
The description of the simple Lie algebra of type $E_7$ in \cite[\S 3.2]{Eld07} can be expressed as follows:
\[
\fre_7=\Bigl(\bigoplus_{i=1}^7\frsl(V_i)\otimes \FF t_i\Bigr)
 \oplus\Bigl(\bigoplus_{\mb{c}\in \mb{C}\setminus\{\mb{0}\}}(V^{\mb{c}}\otimes \FF e^{\mb{c}})\Bigr),
\]
and this gives the unique, up to conjugation by an automorphism, (optimal) short $\mathsf{SL}_2^7$-structure on $\fre_7$. 

The coordinate algebra 
\[
\cC=\Bigl(\bigoplus_{i=1}^7\FF t_i\Bigr)\oplus\Bigl(\bigoplus_{\mb{c}\in \mb{C}\setminus\{\mb{0}\}}
\FF e^{\mb{c}}\Bigr),
\]
that completely determines $\fre_7$ is given by the following equations, that follow from \cite[(3.7)]{Eld07}:
\begin{equation}\label{eq:e7}
\begin{split}
&t_it_j=\delta_{ij}t_i,\quad\text{for $1\leq i,j\leq 7$,}\\[6pt]
&t_ie^{\mb{c}}=e^{\mb{c}}t_i=\begin{cases} e^{\mb{c}}&\text{if $i\in\supp(\mb{c})$,}\\
    0&\text{otherwise,}\end{cases}\\[6pt]
&e^{\mb{c}}e^{\mb{d}}=\epsilon(\mb{c},\mb{d})e^{\mb{c}+\mb{d}},\quad\text{for $\mb{c}\neq\mb{d}\in\mb{C}\setminus\{\mb{0}\}$,}\\[6pt]
&e^{\mb{c}}e^{\mb{c}}=\epsilon(\mb{c},\mb{c})\sum_{i\in\supp(\mb{c})}t_i,\quad\text{for $\mb{c}\in\mb{C}\setminus\{\mb{0}\}$,}
\end{split}
\end{equation}
where $\epsilon(\mb{c},\mb{d})\in\{\pm 1\}$ is the sign that appears in the multiplication table of the real division algebra $\OO$ of the octonions in the basis $\{1,i,j,k,l,il,jl,kl\}$ (see Table \ref{ta:octonions}), under the assignment
\begin{align*}
&&\mb{c_1}&\leftrightarrow i, &\mb{c_2}&\leftrightarrow j,&\mb{c_3}&\leftrightarrow k, \\
\mb{c_4}&\leftrightarrow l,&\mb{c_5}&\leftrightarrow il, &\mb{c_6}&\leftrightarrow jl,&\mb{c_7}&\leftrightarrow kl. &&
\end{align*}
Thus, for instance, $\epsilon(\mb{c}_5,\mb{c}_6)=-1$, because $(il)(jl)=-k$ in $\OO$, so that $e^{\mb{c}_5}e^{\mb{c}_6}=-e^{\mb{c}_5+\mb{c}_6}=-e^{\mb{c}_3}$.

In particular, $\epsilon(\mb{c},\mb{c})=-1$ for any $\mb{c}\in\mb{C}\setminus\{\mb{0}\}$.

\begin{table}[ht!]
{\small
$$
\vbox{\offinterlineskip \halign{\hfil$#$\enspace\hfil\vreglon
 &\hfil\enspace$#$\enspace\hfil\vregla
 &\hfil\enspace$#$\enspace\hfil\vregla
 &\hfil\enspace$#$\enspace\hfil\vregla
 &\hfil\enspace$#$\enspace\hfil&\hskip .2pt\vregla#
 &&\hfil\enspace$#$\enspace\hfil\vregla\cr
 \omit\hfil\vrule width1pt depth 4pt
   &1&i&j&k&\omit\vregleta
     &l&il&jl&kl\cr
 \noalign{\hreglon}
 1&1&i&j&k&\omit\vregleta&l&il&jl&kl\cr
 \noalign{\hregla}
 i&i&-1&k&-j&\omit\vregleta&il&-l&-kl&jl\cr\noalign{\hregla}
 j&j&-k&-1&i&\omit\vregleta&jl&kl&-l&-il\cr\noalign{\hregla}
 k&k&j&-i&-1&\omit\vregleta&kl&-jl&il&-l\cr
 \multispan6{\hregletafill}&\multispan4{\hreglafill}\cr
 l&l&-il&-jl&-kl&&-1&i&j&k\cr\noalign{\hregla}
 il&il&l&-kl&jl&&-i&-1&-k&j\cr\noalign{\hregla}
 jl&jl&kl&l&-il&&-j&k&-1&-i\cr\noalign{\hregla}
 kl&kl&-jl&il&l&&-k&-j&i&-1\cr
 \noalign{\hregla}
}}
$$}\caption{The Octonions}\label{ta:octonions}
\end{table}

\bigskip

\section{Short $\mathsf{SL}_2^8$-structure on $E_8$ and the extended Hamming $[8,4,4]$ binary linear code}\label{se:e8}

The extended Hamming code (Example \ref{ex:extended_Hamming}) is the binary linear code $\mb{H}$ consisting of $\mb{0}$, $\mb{1}$, and the elements
\begin{align*}
\mb{c}_1&=(1,1,0,0,1,1,0,0), & \mb{c}_8&=(0,0,1,1,0,0,1,1)=\mb{1}+\mb{c}_1,\\
\mb{c}_2&=(0,1,1,0,0,1,1,0), & \mb{c}_9&=(1,0,0,1,1,0,0,1)=\mb{1}+\mb{c}_2,\\
\mb{c}_3&=(1,0,1,0,1,0,1,0)=\mb{c}_1+\mb{c}_2, &\mb{c}_{10}&=(0,1,0,1,0,1,0,1)=\mb{1}+\mb{c}_3,\\
\mb{c}_4&=(1,1,1,1,0,0,0,0), & \mb{c}_{11}&=(0,0,0,0,1,1,1,1)=\mb{1}+\mb{c}_4,\\
\mb{c}_5&=(0,0,1,1,1,1,0,0)=\mb{c}_1+\mb{c}_4, & \mb{c}_{12}&=(1,1,0,0,0,0,1,1)=\mb{1}+\mb{c}_5,\\
\mb{c}_6&=(1,0,0,1,0,1,1,0)=\mb{c}_2+\mb{c}_4, & \mb{c}_{13}&=(0,1,1,0,1,0,0,1)=\mb{1}+\mb{c}_6,\\
\mb{c}_7&=(0,1,0,1,1,0,1,0)=\mb{c}_3+\mb{c}_4, & \mb{c}_{14}&=(1,0,1,0,0,1,0,1)=\mb{1}+\mb{c}_7.
\end{align*}

The description of the simple Lie algebra of type $E_8$ in \cite[\S 3.3]{Eld07} can be expressed as follows:
\[
\fre_8=\Bigl(\bigoplus_{i=1}^8\frsl(V_i)\otimes \FF t_i\Bigr)
 \oplus\Bigl(\bigoplus_{\mb{c}\in \mb{H}\setminus\{\mb{0},\mb{1}\}}(V^{\mb{c}}\otimes \FF e^{\mb{c}})\Bigr),
\]
and this gives the unique, up to conjugation by an automorphism, (optimal) short $\mathsf{SL}_2^8$-structure on $\fre_8$. 

The coordinate algebra 
\[
\cH=\Bigl(\bigoplus_{i=1}^8\FF t_i\Bigr)\oplus\Bigl(\bigoplus_{\mb{c}\in \mb{H}\setminus\{\mb{0},\mb{1}\}}
\FF e^{\mb{c}}\Bigr),
\]
is determined by the following equations obtained from \cite[(3.16)]{Eld07}:
\begin{equation}\label{eq:e8}
\begin{split}
&t_it_j=\delta_{ij}t_i,\quad\text{for $1\leq i,j\leq 8$,}\\[6pt]
&t_ie^{\mb{c}}=e^{\mb{c}}t_i=\begin{cases} e^{\mb{c}}&\text{if $i\in\supp(\mb{c})$,}\\
    0&\text{otherwise,}\end{cases}\\[6pt]
&e^{\mb{c}}e^{\mb{d}}=\epsilon(\mb{c},\mb{d})e^{\mb{c}+\mb{d}},\quad\text{for $\mb{c}\neq\mb{d}\in\mb{H}\setminus\{\mb{0}\}$,}\\[6pt]
&e^{\mb{c}}e^{\mb{c}}=\epsilon(\mb{c},\mb{c})\sum_{i\in\supp(\mb{c})}t_i,\quad\text{for $\mb{c}\in\mb{H}\setminus\{\mb{0},\mb{1}\}$,}
\end{split}
\end{equation}
where $\epsilon(\mb{c},\mb{d})\in\{\pm 1\}$. As shown in \cite[\S 3.3]{Eld07}, the signs that appear here are the signs in the multiplication table of $\OO\otimes_\RR \RR[\varepsilon]$ (isomorphic to $\OO\oplus\OO$), where $\RR[\varepsilon]$ is the real algebra $\RR 1\oplus\RR\varepsilon$ with $\varepsilon^2=1$, in the basis
\begin{multline*}
\{1\otimes 1,i\otimes 1,j\otimes 1,k\otimes 1,l\otimes 1, (il)\otimes 1,(jl)\otimes 1,(kl)\otimes 1,\\
1\otimes \varepsilon,
i\otimes \varepsilon,
j\otimes \varepsilon,
k\otimes \varepsilon,
-l\otimes \varepsilon,
-(il)\otimes \varepsilon,
-(jl)\otimes \varepsilon,
-(kl)\otimes \varepsilon\}
\end{multline*}
under the assignment
\begin{align*}
&&\mb{c_1}&\leftrightarrow i\otimes 1, 
&\mb{c_2}&\leftrightarrow j\otimes 1,
&\mb{c_3}&\leftrightarrow k\otimes 1, 
\\
\mb{c_4}&\leftrightarrow l\otimes 1,
&\mb{c_5}&\leftrightarrow (il)\otimes 1, 
&\mb{c_6}&\leftrightarrow (jl)\otimes 1,
&\mb{c_7}&\leftrightarrow (kl)\otimes 1, 
\\
&&\mb{c_8}&\leftrightarrow i\otimes \varepsilon, 
&\mb{c_9}&\leftrightarrow j\otimes \varepsilon,
&\mb{c_{10}}&\leftrightarrow k\otimes \varepsilon, 
\\
\mb{c_{11}}&\leftrightarrow -l\otimes \varepsilon,
&\mb{c_{12}}&\leftrightarrow -(il)\otimes \varepsilon, 
&\mb{c_{13}}&\leftrightarrow -(jl)\otimes \varepsilon,
&\mb{c_{14}}&\leftrightarrow -(kl)\otimes \varepsilon.
\end{align*}

\bigskip

\section{Short $\mathsf{SL}_2^4$-structure on $F_4$}\label{se:f4} 

For $F_4$, the results in \cite[\S 3.1]{Eld07} show that we need to consider the code $\mb{F}\subseteq\FF_2^4$ consisting of the $4$-tuples with an even number of $1$'s, so that a check matrix is the one-row matrix $\begin{pmatrix} 1&1&1&1\end{pmatrix}$. The codewords are $\mb{0}$, $\mb{1}$, and
\begin{align*}
\mb{c}_1&=(1,1,0,0),\quad&\mb{c}_4&=(0,0,1,1)=\mb{1}+\mb{c_1},\\
\mb{c}_2&=(0,1,1,0),\quad&\mb{c}_5&=(1,0,0,1)=\mb{1}+\mb{c_2},\\
\mb{c}_3&=(1,0,1,0),\quad&\mb{c}_6&=(0,1,0,1)=\mb{1}+\mb{c_3}.
\end{align*}
Then we have
\[
\frf_4=\Bigl(\bigoplus_{i=1}^4\frsl(V_i)\otimes \FF t_i\Bigr)
 \oplus\Bigl(\bigoplus_{\mb{c}\in \mb{F}\setminus\{\mb{0}\}}(V^{\mb{c}}\otimes \FF e^{\mb{c}})\Bigr),
\]
and the coordinate algebra is the ``code-like'' algebra
\[
\cF=\Bigl(\bigoplus_{i=1}^4\FF t_i\Bigr)\oplus\Bigl(\bigoplus_{\mb{c}\in\mb{F}\setminus\{\mb{0}\}}\FF e^{\mb{c}}\Bigr),
\]
with 
\begin{equation}\label{eq:f4}
\begin{split}
&t_it_j=\delta_{ij}t_i,\quad\text{for $1\leq i,j\leq 4$,}\\[6pt]
&t_ie^{\mb{c}}=e^{\mb{c}}t_i=\begin{cases} e^{\mb{c}}&\text{if $i\in\supp(\mb{c})$,}\\
    0&\text{otherwise,}\end{cases}\\[6pt]
&e^{\mb{c}}e^{\mb{d}}=\epsilon(\mb{c},\mb{d})e^{\mb{c}+\mb{d}},\quad\text{for $\mb{c}\neq\mb{d}\in\mb{F}\setminus\{\mb{0}\}$,}\\[6pt]
&e^{\mb{c}}e^{\mb{c}}=\epsilon(\mb{c},\mb{c})\sum_{i\in\supp(\mb{c})}t_i,\quad\text{for $\mb{c}\in\mb{F}\setminus\{\mb{0},\mb{1}\}$,}
\end{split}
\end{equation}
where the scalars $\varepsilon(\mb{c},\mb{d})$ are given in Table \ref{ta:f4}.

\begin{table}[ht!]
$$
\vbox{\offinterlineskip \halign{\hfil$#$\hfil\enspace\vreglon
 &\enspace\hfil$#$\hfil\enspace\vregla
 &\enspace\hfil$#$\hfil\enspace\vregla
 &\enspace\hfil$#$\hfil\enspace\vregla
 &\enspace\hfil$#$\hfil\enspace&\hskip .2pt\vregla#
 &&\enspace\hfil$#$\hfil\enspace\vregla\cr
 \omit\enspace\hfil\raise2pt\hbox{$\ {}$}%
  \enspace\hfil\vrule width1pt depth 4pt
   &\enspace\mb{0}\enspace&\enspace\mb{c}_1\enspace&\enspace\mb{c}_2\enspace
   &\enspace\mb{c}_3\enspace&\omit\vregleta
     &\enspace\mb{1}\enspace
     &\enspace\mb{c}_4\enspace
     &\enspace\mb{c}_5\enspace&\enspace\mb{c}_6\enspace\cr
 \noalign{\hreglon}
 \mb{0}&1&1&1&1&\omit\vregleta&1&1&1&1\cr
 \noalign{\hregla}
 \mb{c}_1&1&-2&1&1&\omit\vregleta&1&-2&-1&-1\cr\noalign{\hregla}
 \mb{c}_2&1&1&-2&1&\omit\vregleta&1&-1&-2&-1\cr\noalign{\hregla}
 \mb{c}_3&1&1&1&-2&\omit\vregleta&1&-1&-1&-2\cr
 \multispan6{\hregletafill}&\multispan4{\hreglafill}\cr
 \mb{1}&1&-1&-1&-1&&-1&1&1&1\cr\noalign{\hregla}
 \mb{c}_4&1&2&-1&-1&&-1&-2&-1&-1\cr\noalign{\hregla}
 \mb{c}_5&1&-1&2&-1&&-1&-1&-2&-1\cr\noalign{\hregla}
 \mb{c}_6&1&-1&-1&2&&-1&-1&-1&-2\cr
 \noalign{\hregla}
}}
$$
\caption{\null}\label{ta:f4}
\end{table}

Note that, unlike code algebras, the codeword $\mb{1}$ is included in the definition of $\cF$.

\bigskip

\section{Optimal short $\mathsf{SL}_2^n$-structures on classical Lie algebras}\label{se:classical}

As a final remark, the coordinate algebras for the optimal short $\mathsf{SL}_2^n$-structures on the other simple Lie algebras in Corollary \ref{co:optimal} are easy to deduce. We indicate here how to proceed and leave the details to the reader.

For $A_1$ the situation is trivial. The simple Lie algebra of type $C_n$ is the symplectic Lie algebra on an orthogonal sum of $2$-dimensional vector spaces: $\frsp(V_1\perp\cdots\perp V_n)$, with $V_1,\ldots,V_n$ $2$-dimensional vector spaces endowed with a nonzero skew symmetric bilinear form $\langle .\mid.\rangle$. We get the decomposition (note that $\frsl(V_i)=\frsp(V_i)$):
\[
\frc_n=\frsp(V_1\perp\cdots\perp V_n)=\Bigl(\bigoplus_{i=1}^n\frsl(V_i)\Bigr)\oplus\Bigl(\bigoplus_{1\leq i<j\leq n}(V_i\otimes V_j)\Bigr),
\]
where $u\otimes v\in V_i\otimes V_j$ is identified with the linear map
\[
w\mapsto\begin{cases} \langle w\mid u\rangle v&\text{if $w\in V_i$,}\\
  \langle w\mid v\rangle u&\text{if $w\in V_j$,}\\
  0&\text{otherwise.}
  \end{cases}
\]
This gives the (optimal) short $\mathsf{SL}_2^n$-structure on $\frc_n$.

\smallskip

A bit more involved is the case $D_{2n}$. Let $V_1,\ldots,V_{2n}$ be as above, then for each $i=1,\ldots,n$, the $4$-dimensional vector space $V_{2i-1}\otimes V_{2i}$ is endowed with the nondegenerate symmetric bilinear form given by
\[
\bigl(u\otimes v\mid u'\otimes v'\bigr)=\langle u\mid u'\rangle\langle v\mid v'\rangle,
\]
and we may identify the simple Lie algebra of type $D_{2n}$ with the orthogonal Lie algebra $\frso\bigl(V_1\otimes V_2\perp \cdots\perp V_{2n-1}\otimes V_{2n}\bigr)$.

Given a vector space $W$ endowed with a nondegenerate symmetric bilinear form $b(.,.)$, the orthogonal Lie algebra $\frso(W)$ is spanned by the linear maps:
\[
\sigma_{w_1,w_2}:w\mapsto b(w,w_1)w_2-b(w,w_2)w_1.
\]
Hence, if $W$ is the orthogonal direct sum, relative to $b$, $W=W_1\perp W_2$, then $\frso(W)$ decomposes as $\frso(W_1)\oplus\frso(W_2)\oplus \sigma_{W_1,W_2}$, and we may identify $\sigma_{W_1,W_2}$ with $W_1\otimes W_2$.

Now, $\frso(V_{2i-1}\otimes V_{2i})$ is naturally isomorphic to $\frsl(V_{2i-1})\oplus\frsl(V_{2i})$. It follows that we may describe $\frd_{2n}=\frso\bigl(V_1\otimes V_2\perp\cdots\perp V_{2n-1}\otimes V_{2n}\bigr)$ as
\[
\frd_{2n}=\Bigl(\bigoplus_{i=1}^{2n}\frsl(V_i)\Bigr)\oplus
\Bigl(\bigoplus_{1\leq i<j\leq n}(V_{2i-1}\otimes V_{2i}\otimes V_{2j-1}\otimes V_{2j})\Bigr),
\]
thus obtaining the (optimal) short $\mathsf{SL}_2^{2n}$-structure on $\frd_{2n}$.

\smallskip

In the same vein, $B_{2n}$ may be identified with $\frso\bigl(\FF\perp V_1\otimes V_2\perp\cdots \perp V_{2n-1}\otimes V_{2n}\bigr)$, and from here one gets the decomposition
\[
\frb_{2n}=\Bigl(\bigoplus_{i=1}^{2n}\frsl(V_i)\Bigr)\oplus
\Bigl(\bigoplus_{1\leq i\leq n}(V_{2i-1}\otimes V_{2i})\Bigr)\oplus
\Bigl(\bigoplus_{1\leq i<j\leq n}(V_{2i-1}\otimes V_{2i}\otimes V_{2j-1}\otimes V_{2j})\Bigr),
\]
obtaining the (optimal) short $\mathsf{SL}_2^{2n}$-structure on $\frb_{2n}$.

The computation of the coordinate algebras in these cases is straightforward.

\bigskip

\section{Final thoughts}

As mentioned in Section \ref{se:Sstructures}, $\mathsf{S}$-structures generalize gradings by abelian groups on algebras, in the sense that a morphism $\mathsf{Q}\rightarrow \AAut(\cA)$, for a quasitorus $\mathsf{Q}$ and a nonassociative algebra $\cA$, is substituted by a morphism $\mathsf{S}\rightarrow \AAut(\cA)$, for an arbitrary reductive group $\mathsf{S}$. This is reflected in the title of \cite{Vinberg}: \emph{Non-abelian gradings of Lie algebras}. 

In a grading, the algebra $\cA$ splits into a direct sum of one-dimensional irreducible submodules for the corresponding quasitorus $\mathsf{Q}$, while in an $\mathsf{S}$-structure there are more options for these irreducible modules.

Fine gradings (see \cite[\S 1.3]{EKmon}) are gradings where the irreducible modules for $\mathsf{Q}$ appear with low multiplicity, that is, the homogeneous components are small. In this work, the optimal $\mathsf{SL}_2^n$-structures that have been considered satisfy that the multiplicities of the irreducible modules that appear are all equal to $1$, so they can be thought as a sort of `fine $\mathsf{S}$-structures'. The outcome is that some nice descriptions of several simple Lie algebras are obtained, showing an intriguing connection with code algebras.

The notion of $\mathsf{S}$-structure opens a broad area of research. It also extends the notion of Lie algebras graded by finite root systems, which may be seen as algebras with particular $\mathsf{S}$-structures.

The interesting $\mathsf{S}$-structures to be studied should be based on a nice group $\mathsf{S}$, with an isotypic decomposition of $\cA$ leading to a not too big, or not too complicated, coordinate algebra.

\bigskip


\end{document}